\documentclass[12pt,reqno]{amsart}
\usepackage{amsmath}
\usepackage{amssymb,amscd}
\usepackage{multicol}

\numberwithin{equation}{section}

\newtheorem{theorem}{Theorem}[section]
\newtheorem{corollary}[theorem]{Corollary}
\newtheorem{lemma}[theorem]{Lemma}

\theoremstyle{definition}
\newtheorem{defn}[theorem]{Definition}

\def \mc{\mathcal}

\def \({\left(}
\def \){\right)}
\def \<{\langle}
\def \>{\rangle}

\def \deg{\mathrm{deg}}

\begin{document}

\title[On stability of toric tangent bundles]{On stability of tangent bundle of toric varieties}

\author[I. Biswas]{Indranil Biswas}
\address{School of Mathematics, Tata Institute of Fundamental Research, Mumbai 400005, India}

\email{indranil@math.tifr.res.in}

\author[A. Dey]{Arijit Dey}

\address{Department of Mathematics, Indian Institute of Technology-Madras, Chennai, India}

\email{arijitdey@gmail.com}

\author[O. Genc]{Ozhan Genc}

\address{Middle East Technical University, Northern Cyprus Campus, Guzelyurt, Mersin 10, Turkey.
\textit{Current Address: Department of Mathematics and Informatics, Jagiellonian University, {\L}ojasiewicza 6, 30-348 Krak{\'o}w, Poland}}

\email{ozhangenc@gmail.com}

\author[M. Poddar]{Mainak Poddar}

\address{Mathematics Department, Indian Institute of Science Education and Research, Pune, India}

\email{mainakp@gmail.com}



\begin{abstract} 
Let $X$ be a nonsingular complex projective toric variety. We address the question of semi-stability as well 
as stability for the tangent bundle $T{X}$. In particular, a complete answer is given
when $X$ is a Fano toric variety of dimension four with Picard number at most two,
complementing earlier work of Nakagawa. 
We also give an infinite set of examples of Fano toric varieties  for which
$TX$ is unstable; the dimensions of this collection of varieties are unbounded. Our 
method is based on the equivariant approach initiated by Klyachko and developed 
further by Perling and Kool. 
\end{abstract}

\maketitle

\section{Introduction}

Let $X$ be a smooth complex projective variety. If the canonical line bundle $K_{X}$ is ample, then
from a theorem of Yau, \cite{Ya}, and Aubin, \cite{Au}, it follows that the tangent bundle 
$T{X}$ is semistable (in the sense of Mumford and Takemoto) with respect to the polarization 
$K_X$. The variety $X$ is said to be Fano if the anti-canonical line bundle $K_{X}^{-1}$ is 
ample. Fano varieties are very basic objects in birational classification of complex algebraic 
varieties (minimal model program), for example a theorem of Birkar-Cascini-Hacon-McKernan \cite{BCCM} says that 
every uniruled variety is birational to a variety which has a fibration with a Fano general 
fiber. Stability of the tangent bundle of a nonsingular Fano variety with respect to polarization 
$K_{X}^{-1}$ is a question of interest originated mostly from a differential geometric point of 
view. The existence of Einstein-K$\ddot{\text{a}}$hler  metric on $X$ implies the polystability of the 
tangent bundle with respect to $K_X^{-1}$ \cite{Kob, Lub}. In general converse of this result is not true. The simplest example is the surface $\Sigma_2$ obtained by blowing up the complex plane 
$\mathbb{P}^2$ at two points. In this paper we are interested in studying semi-stability as well 
as stability of the tangent bundle $TX$ when $X$ is a toric variety and in particular when $X$ is a 
Fano toric variety.

Let $X$ be a nonsingular complex projective toric variety of dimension $n$, equipped with
an action of the
$n$--dimensional complex torus $T$. A coherent torsion-free sheaf $\mc{E}$ on $X$ is said to be 
$T$--equivariant (or $T$--linearized) if it admits a lift of the $T$--action on $X$, which is 
linear on the stalks of $\mc{E}$. Fix a polarization $H$ of $X$, where $H$ is a $T$--equivariant 
very ample line bundle (equivalently, $T$--invariant very ample divisor) of $X$.

A $T$--equivariant coherent torsion-free sheaf $\mc{E}$ on $X$ is said to be equivariantly 
stable (respectively, equivariantly semistable) if $\mu(\mc{F}) \,<\, \mu(\mc{E})$ (respectively, 
$\mu(\mc{F})\, \le\, \mu(\mc{E})$) for every proper $T$--equivariant proper subsheaf $\mc{F}\,\subset 
\,\mc{E}$ (see Section 2). From the uniqueness of the 
Harder-Narasimhan filtration it follows easily that the notions of semi-stability and equivariant 
semi-stability of an equivariant torsion-free sheaf on a nonsingular toric variety are 
equivalent. Further, in case of equivariant torsion-free sheaves, the notions of
equivariant stability and stability coincide (see Theorem \ref{stability} or \cite[Proposition 4.13]{Kool2}). Using this equivariant 
approach, we investigate the stability and semi-stability of the tangent bundle of a nonsingular 
toric variety.

Our main results are as follows.

\begin{enumerate}
\item Determination of the stability (or otherwise) of the tangent bundle of Hirzebruch surfaces 
for an arbitrary polarization; see Theorem \ref{HS} and Corollary \ref{HS2}.

\item A very simple proof of the well-known result that 
$T\mathbb{P}^n$ is stable with respect to the anti-canonical
polarization (Theorem \ref{projective spaces}).

\item We identify all nonsingular Fano toric $4$-folds  with Picard number at most two that have semi-stable tangent bundle (Theorem \ref{fano4}). In particular,  we get an example of a Fano toric $4$-fold (namely, $B_5$) which has a strictly semi-stable tangent bundle, but does not admit Einstein-K$\ddot{\text{a}}$hler metric (cf. \cite{Nak2}).  
 
\item Construction of an infinite family of Fano toric varieties with unstable tangent bundle, 
consisting of $\mathbb{P}(\mc{O}_{\mathbb{P}^n} \oplus \mc{O}_{\mathbb{P}^n} (m))$ for all $n\ge 
2$ and $m \le n$ (Theorem \ref{family}). The case $n=2$ was settled earlier by
Steffens in \cite{Stef}.
\end{enumerate}

The general strategy here is as follows. We use the isotypical decomposition of an equivariant 
sheaf to describe it in terms of certain combinatorial data, following Perling \cite{Per} and 
Kool \cite{Kool1}. (Of course, both draw inspiration from the seminal work of Klyachko 
\cite{Kly}.) 
 We prove a formula in Lemma \ref{rank} that calculates the rank of an equivariant torsion--free coherent 
sheaf on $X$ from the combinatorial data. With a little bit of fine-tuning, this specializes to a very useful 
rank formula, see \eqref{rank2}, for an arbitrary equivariant coherent subsheaf of the tangent bundle. We obtain a similar type of formula for the 
degree of such a subsheaf, see \eqref{degf}, using a formula of Kool for the first Chern class. 
Using these formulas, we can identify the combinatorial data that may be associated to a subsheaf 
of a given rank whose slope exceeds that of the tangent bundle; see Lemma \ref{collection} and 
Lemma \ref{matrix}. We then examine if a subsheaf with the given rank and corresponding to such 
combinatorial data really exists by studying the transition maps associated to the combinatorial 
data.

\section{From equivariant stability to stability}

Given a coherent torsion-free sheaf $\mc{E}$ on a projective variety $X$ of dimension $n$, the slope $\mu(\mc{E})$ with respect to a polarization $H$ on $X$ is defined as the ratio 
$$ \mu(\mc{E}) = \frac{ \deg\,{\mc{E}} }{ {\rm rank}\,{\mc{E}}} \,,$$ where the degree of $\mc{E}$ is defined as the intersection product
$\deg\,{\mc{E}}\,:=\, c_1(\mc{E})\cdot H^{n-1}$. A subsheaf $\mc F$ of $\mc E$ is said to be a proper subsheaf if $0 \,< \,\text{rank}(\mc F) \,<\,\text{rank}(\mc E)$. A torsion-free sheaf $\mc{E}$ is said to be $\mu$--stable (respectively, $\mu$--semistable) if 
 $ \mu(\mc{F}) < \mu({\mc{E}}) $ (respectively, $ \mu(\mc{F}) \le \mu({\mc{E}}) $) for every proper subsheaf $\mathcal{\mc F} \,\subset\, \mc E$.
 The notion of $\mu$--stability (semistability) was first introduced by Mumford and Takemoto. In this article,
 (semi)stability we will always mean $\mu$--(semi)stability, unless otherwise specified. Also, a sheaf $\mc{E}$ will be called unstable if $\mc{E}$ is not semistable.
 
Stable and semistable sheaves play an important role in the structure theory of coherent sheaves 
(cf. \cite{HL}). Every torsion-free coherent sheaf admits the Harder-Narasimhan filtration such 
that each successive quotient is semistable. A semistable sheaf, in turn, admits a
Jordan-Holder filtration such that each successive quotient is stable of same slope.

In this section we give a proof of the crucial fact that for an equivariant torsion-free sheaf on 
a nonsingular toric variety, equivariant stability is equivalent to usual stability. This 
result is also proved by Kool \cite[Proposition 4.13]{Kool2} for reflexive sheaves. The proof 
given here is different.

\begin{theorem}\label{stability}
Let $\mathcal{E}$ be an equivariant torsion-free sheaf on a projective toric variety $X$. Then 
$\mathcal{E}$ is equivariantly stable if and only if 
$\mathcal{E}$ is stable. 
\end{theorem}

\begin{proof} 
	If $\mathcal{E}$ is stable, then it is evidently equivariantly stable.
	We will prove that $\mathcal{E}$ is stable if it is equivariantly stable.
	
	We first note that it is enough to prove this under the extra assumption that $\mathcal{E}$ is reflexive. Indeed,
	if $\mathcal{E}$ is torsion-free and equivariantly stable, then $\mathcal{E}^{\vee\vee}$ is reflexive and
	equivariantly stable. On the other hand, if $\mathcal{E}^{\vee\vee}$ is stable, then clearly
	$\mathcal{E}$ is also stable. Indeed, for any coherent subsheaf $\mc{F}\, \subset\, \mathcal{E}^{\vee\vee}$, we have
	$\deg\, \mc{F}\,=\, \deg (\mc{F}\cap \mathcal{E})$.
	
	So assume that $\mathcal{E}$ is reflexive and equivariantly stable. We will first show that 
	$\mathcal{E}$ is semistable.
	
	To prove semistability by contradiction, assume that $\mathcal{E}$ is not semistable. Then
	there is a unique maximal destabilizing semistable subsheaf
	\begin{equation}\label{j2}
	\mathcal{F}\, \subset\,  \mathcal{E}\, .
	\end{equation}
	In other words, $\mathcal{F}$ is the smallest nonzero subsheaf of $\mathcal{E}$ in the
	Harder-Narasimhan filtration of $\mathcal{E}$.
	
	As before, $T\, \subset\, \text{Aut}(X)$ is the torus acting on $X$.
	Let $$\Phi\,:\, T\times \mathcal{E}\, \longrightarrow\, \mathcal{E}$$ be an action of the torus $T$ on
	$\mathcal{E}$ lifting the action of $T$ on $X$. For any element $t\, \in\, T$, the homomorphism
	$$\mathcal{E}\, \longrightarrow\, \mathcal{E}\, , \ \ v\, \longmapsto\, \Phi(t,\, v)$$
	will be denoted by $\Phi_t$. Note that $\Phi_t$ is an automorphism of $\mathcal{E}$ over the automorphism
	of $X$ given by the action of $t$ on it. For any $t\, \in\, T$, and any coherent subsheaf
	$\mc{V}\, \subset\, \mathcal{E}$, the coherent subsheaf $\Phi_t(\mc{V})\, \subset\, \mathcal{E}$ will
	be denoted by $t\cdot \mc{V}$. The above automorphism $\Phi_t$ produces an isomorphism
	\begin{equation}\label{j1}
	(t^{-1})^*\mc{V}\, \stackrel{\sim}{\longrightarrow}\,  t\cdot \mc{V}
	\end{equation}
	over the identity map of $X$. 
	
	Since $\mu(t^{\ast}\mc{V})\,=\, \mu(\mc{V})$ for any coherent sheaf $\mc{V}$ on $X$, from \eqref{j1}
	it follows that $\mu(t\cdot \mc{V})\,=\, \mu(\mc{V})$ for every coherent subsheaf $\mc{V}\, \subset\, \mathcal{E}$.
	This and \eqref{j1} together imply that the subsheaf $\mc{F}$ in \eqref{j2} has the property that
	$t\cdot \mc{F}$ is also a maximal destabilizing subsheaf of $\mc{E}$. Therefore, from the
	uniqueness of the maximal destabilizing subsheaf it is deduced that
	$$t\cdot \mc{F}\,=\, \mc{F}\, \subset\, \mc{E}\, .$$
	Consequently, the action of $T$ on $\mathcal{E}$ preserves the subsheaf $\mc{F}$.
	Therefore, $\mc{F}$ is equivariant. This contradicts the given condition that $\mc{E}$ is equivariantly semistable.
	Hence we conclude that $\mc{E}$ is semistable.
	
	Let $\mc{H}$ be the socle of $\mc{E}$; in other words, $\mc{H}$ is the maximal polystable 
	subsheaf of $\mc{E}$ with the same slope as $\mc{E}$ \cite[Proposition 3.1]{BDL}, \cite[p. 23, Lemma 1.5.5]{HL}. From
	the uniqueness of $\mc{H}$, it follows that $t\cdot\mc{H} \,=\, \mc{H}$ 
	for every $t \,\in\, T$. Therefore, $\mc H$ is a $T$--equivariant subsheaf of
	$\mc{E}$. Since $\mc{E}$ is equivariantly stable, and the slopes of $\mc H$ and $\mc{E}$
	coincide, we must have $\mc{H}\,=\, \mc{E}$. This implies that $\mc{E}$ is polystable. 
	
	Since $\mc{E}$ is polystable, it suffices to show that $\mc{E}$ is indecomposable. Note that
	$\mc{E}$ is indecomposable if the dimension of a maximal torus in the algebraic group
	$\text{Aut}(\mc{E})$ is one \cite[p.~201, Proposition 16]{At2}; the automorphism
	group $\text{Aut}(\mc{E})$ is a Zariski open subset of the affine space $H^0(X,\, \text{End}(\mc{E}))$.
	
	The action of $T$ on $\mc{E}$ produces an action of $T$ on the group $\text{Aut}(\mc{E})$:
	$$
	(t\cdot A)(v) \,=\, t\cdot A(t^{-1}\cdot v)\, , \ \ \forall \ \ A\, \in\, \text{Aut}(\mc{E})\, , \ v\, \in\, \mc{E}\, ,
	$$
	for every $t\, \in\, T$.
	We will show that there is a maximal torus $\widetilde{T}\, \subset\, \text{Aut}(\mc{E})$ on which
	$T$ acts trivially. For this, first consider the semi-direct product
	$\text{Aut}(\mc{E})\rtimes T$ for this action of $T$ on $\text{Aut}(\mc{E})$. Let
	$\widetilde{T}'\, \subset\, \text{Aut}(\mc{E})\rtimes T$ be a maximal torus containing
	the subgroup $T$ of $\text{Aut}(\mc{E})\rtimes T$.
	Then
	$$
	\widetilde{T}\, :=\, \widetilde{T}'\cap \text{Aut}(\mc{E})\, \subset\, \text{Aut}(\mc{E})
	$$
	is a maximal torus of $\text{Aut}(\mc{E})$ on which $T$ acts trivially. Now
	$\widetilde{T}$ produces an eigenspace decomposition of $\mc{E}$ for the characters of $\widetilde{T}$
	$$
	\mc{E}\,=\, \bigoplus_{\chi\in \widetilde{T}^*} \mc{E}_\chi\, ;
	$$
	any $t\, \in\, \widetilde{T}$ acts on $\mc{E}_\chi$ as multiplication by $\chi(t)$.
	The direct summands in this decomposition
	are preserved by the action of $T$ on $\mc{E}$, because $T$ acts trivially on $\widetilde{T}$
	(see \cite[p.~55, Proposition~1.2]{BP} for a general result). But $\mc{E}$ is equivariantly stable, so it does
	not admit any nontrivial $T$--equivariant decomposition. This implies that $\dim \widetilde{T}\,=\, 1$, because
	that action of $\widetilde{T}$ on $\mc{E}$ is faithful. As noted before, this implies that $\mc{E}$ is indecomposable.
\end{proof}

\section{Equivariant coherent sheaves on $X$} 

We briefly review the classification of equivariant coherent sheaves on a nonsingular toric variety following Perling \cite{Per}. The notation established in this section will be used extensively
in the rest of the paper.

Let $X$ be a nonsingular complex projective toric variety of dimension $n$, equipped with the action of an $n$--dimensional torus $T$. Let $M$ and 
$N$ denote the group of characters of $T$ and the group of one--parameter subgroups of $T$ respectively. Then both $M$ and $N$ are free $\mathbb 
Z$--modules of rank $n$ that are naturally dual to one another. Let $\Delta$ denote the fan of $X$. It is a collection of rational cones in the real 
vector space $N \otimes_{\mathbb Z} \mathbb{R}$ closed under the operations of taking faces, and performing intersections. Denote the set of 
$d$-dimensional cones of the fan $\Delta$ of $X$ by $\Delta(d)$. For any 
one--dimensional cone (ray) $\alpha \,\in\, \Delta(1)$, its primitive co-character generator is also denoted by $\alpha$; this is for notational 
convenience. We refer the reader to \cite{Ful,Oda} for details on toric varieties.

Let $\mc{E}$ be a $T$--equivariant coherent sheaf over $X$ of rank $r$. 
Let $X_{\sigma}$ be any affine toric subvariety of $X$ corresponding to a cone $\sigma$. Denote by $S_{\sigma}$ the semigroup
$$\{m \,\in\, M \, \mid\, \langle m, \,\alpha \rangle \,\ge\, 0 \ \ \forall \ \alpha \,\in\, \sigma \} \,\subseteq\, M\, .$$ Let $k[S_{\sigma}]$ be the finitely generated semigroup algebra which is the coordinate ring of 
$X_{\sigma}$. Let $E^{\sigma}$ denote the 
$k[S_{\sigma}]$--module $\Gamma(X_{\sigma},\, \mc{E})$ consisting of sections of $\mc{E}$ over $X_{\sigma}$. Consider the isotypical decomposition
\begin{equation}\label{isotyp}
E^{\sigma} \,=\, \bigoplus_{m \in M} E^{\sigma}_m \,.
\end{equation}
For any $m' \,\in\, M $ with $\chi(m')\,\in\, k[S_{\sigma}]$, there is a natural multiplication
map
\begin{equation}\label{chim}
\chi(m')\,:\, E^{\sigma}_m \,\longrightarrow\, E^{\sigma}_{m+m'} \,.
\end{equation}
This homomorphism is injective if $\mc{E}$ is torsion-free.

Let \begin{equation}\label{ssperp}
S_{\sigma^{\perp}} \,=\, \{m \,\in\, M\,\mid\, \langle m,\, \alpha \rangle \,=\, 0 \ \ \forall \ \alpha \,\in\, \sigma \} \, . \end{equation}
Define 
\begin{equation}
M_{\sigma} \,= \,M/S_{\sigma^{\perp}} \,.
\end{equation} Denote by $[m]$ the equivalence class 
of $m$ in $M_{\sigma}$. Note that $M_{\sigma}$ may be identified with the character group of an appropriate subtorus $T_{\sigma}$ of $T$,
namely the maximal subtorus of $T$ that has a fixed point in $X_{\sigma}$.
Let $A_{\sigma}$ be the subvariety of $X_{\sigma}$ defined by the ideal generated by the set $\{ \chi(m) -1 \,\mid\, m \,\in\, S_{\sigma^{\perp}}\}$. Then 
$A_{\sigma}$ has a dense $T_{\sigma}$--orbit. So elements of $M_{\sigma}$ generate the field of rational functions on $A_{\sigma}$. 

If $m' \in S_{\sigma^{\perp}}$,
then $\chi(m')$ in \eqref{chim} is an isomorphism.
 Denote the isomorphism class of $E^{\sigma}_{m'}$, for $m' \in m + S_{\sigma^{\perp}} $, by $E^{\sigma}_{[m]}$.
 The space $E^{\sigma}_{[m]}$ may be identified with the space of sections of the $T_{\sigma}$--equivariant bundle 
 $\mathcal E|_{A_{\sigma}}$ of weight $[m] \in M_{\sigma}$.
 We have, in fact, an isotypical decomposition 
 \begin{equation}\label{isod} \Gamma(A_{\sigma}, \mathcal{E}) = \bigoplus_{[m] \in M_{\sigma}} E^{\sigma}_{[m]} \,. \end{equation}

Moreover, for $[m] \,\in \,M_{\sigma}$ and any $m' \,\in\, M $ such that $\chi(m')\,\in\, k[S_{\sigma}]$, the map $\chi(m')$ in \eqref{chim} induces a
 map 
 \begin{equation}\label{chim2}
\chi^{\sigma}([m']): E^{\sigma}_{[m]} \longrightarrow E^{\sigma}_{[m+m']} \,.
\end{equation}
 Here, we may naturally identify $\chi^{\sigma}([m'])$ with a character of $T_{\sigma}$. 
 
\begin{defn}\label{order} Define an equivalence relation $\leq_{\sigma}$ on $M$ by setting $m \leq_{\sigma} m'$ if and only if $m'-m \in S_{\sigma}$. This yields a directed pre-order on $M$ which is a partial order when $\sigma$ is of top dimension.

If $m \,\leq_{\sigma}\, m'$, but $m' \,\leq_{\sigma}\, m$ does not hold, we say that $m \,<_{\sigma}\, m'$. \end{defn}

\begin{defn} Let $\{E^{\sigma}_m \,\mid\, m \,\in\, M\}$ be a family of 
 $k$--vector spaces. For each relation $m \,\leq_{\sigma}\, m'$, let
there be given a $k$--linear map $$\chi^{\sigma}(m,m') \,:\, E^{\sigma}_m \,
\longrightarrow\, E^{\sigma}_{m'}$$ such that $ \chi^{\sigma}(m,m)\,=\,1$ for
all $m \,\in\, M$, and also
$$\chi^{\sigma}(m,m'') \,=\,
\chi^{\sigma}(m',m'') \circ \chi^{\sigma}(m,m')$$ for all triples $m
\,\leq_{\sigma}\, m' \, \leq_{\sigma} \,m''$. 
We refer to the $\chi^{\sigma}(m,m'') $'s as multiplication maps.
Denote such data by 
$\widehat{E}^{\sigma}$ and call it a  $\sigma$-family. A
morphism ${\phi}^{\sigma}\,:\, \widehat{E}^{\sigma} \,\longrightarrow \,
\widehat{E}'^{\sigma} $ of 
$\sigma$--families is given by a collection of linear maps
$\{ \phi_m^{\sigma} \,:\, E^{\sigma}_m \,\longrightarrow\, E'^{\sigma}_m\,
\mid\, m \,\in\, M \}$ respecting the multiplication maps.
\end{defn}

\begin{defn}
	A  $\sigma$-family  $\widehat{E}^{\sigma}$ is called  finite if 
	\begin{enumerate}
		\item all the $k$--vector spaces $E^{\sigma}_m$ are finite dimensional,
		\item  for each chain $ \ldots   <_{\sigma} m_{i-1} <_{\sigma} m_i  <_{\sigma} \ldots  $ of elements of $M$,  there exists an $i_0 \in \mathbb{Z}$ such that $E^{\sigma}_{m_i} = 0$ for all $i < i_0$, and
		\item there are only finitely many vector spaces $E^{\sigma}_m $ such that the map
		$$ \bigoplus_{m' <_{\sigma} m } E^{\sigma}_{m'} \longrightarrow E^{\sigma}_m $$
 defined by the summation of $\chi^{\sigma} (m', m)$'s is not surjective.
	\end{enumerate} 
	 
\end{defn}
\begin{theorem}[{\cite{Per}}]
The category of $T$--equivariant coherent sheaves on $X_{\sigma}$ is equivalent to the category of finite  $\sigma$--families $\{ \widehat{E}^{\sigma} \}$.
\end{theorem}

Let $\tau \le \sigma$ be a subcone. Let $i_{\tau,\sigma} : X_{\tau} \longrightarrow X_{\sigma} $ be the corresponding inclusion map. 
Define
$$i_{\tau\sigma}^* (E^{\sigma}) = E^{\sigma} \otimes_{k[S_{\sigma}]} k[S_{\tau}] \,.$$ Note that 
$i_{\tau\sigma}^* (E^{\sigma}) $ has a natural $M$--grading.

\begin{defn}\label{deltafam} Let $\Delta$ be a fan. A collection
$\{ \widehat{E}^{\sigma}\,\mid\, \sigma \,\in\, \Delta\}$ of finite  $\sigma$--families is called a finite  $\Delta$--family, denoted 
$\widehat{E}^{\Delta}$, if for every pair $\tau < \sigma $, there is an
isomorphism 
$\eta_{\tau\sigma} \,:\, i_{\tau\sigma}^* \widehat{E}^{\sigma} \,\longrightarrow
\, \widehat{E}^{\tau}$, such that for every triple 
$\rho \,<\,\tau \,<\,\sigma $, the following holds:
$$\eta_{\rho\sigma}  \,=\, \eta_{\rho\tau} \circ i_{\rho\tau}^* \, \eta_{\tau\sigma}\, .$$

A morphism of finite $\Delta$--families is a collection of morphisms
$$\{ \phi^{\sigma}\,:\, \widehat{E}^{\sigma} \,\longrightarrow\,
\widehat{E}'^{\sigma} \,\mid\, \sigma \,\in\, \Delta \}$$ of finite
 $\sigma$--families such that for all $\tau < \sigma$ the following diagram commutes:
 $$
 \begin{CD}
i_{\tau\sigma}^* (\widehat{E}^{\sigma}) @> i_{\tau\sigma}^* (\phi^{\sigma}) >> i_{\tau\sigma}^* (\widehat{E}'^{\sigma}) \\
@V \eta_{\tau \sigma} VV @V \eta'_{\tau \sigma} VV \\
 \widehat{ E}^{\tau} @> \phi^{\sigma} >> \widehat{E}'^{\tau} \end{CD} $$
\end{defn}

Since $S_{\sigma^{\perp}} \subseteq S_{\tau^{\perp}}$, there is a surjective group homomorphism $$ M/S_{\sigma^{\perp}} \longrightarrow 
 M/S_{\tau^{\perp}}\,.$$ Then $\eta_{\tau \sigma}$ induces an isomorphism 
 $\eta_{\tau \sigma} : ( i_{\tau\sigma}^* (E^{\sigma}_{[m]})) \longrightarrow E^{\tau}_{[m]}$. 
 
\begin{theorem}[{\cite{Per}}] The category of finite  $\Delta$--families is equivalent to the category of coherent $T$--equivariant sheaves 
over $X$.
\end{theorem}

For a $T$--equivariant subsheaf $\mc{F}$ of $\mc{E}$, one has ${F}^{\sigma}_m \subseteq {E}^{\sigma}_m $ for every $\sigma \in \Delta$ and 
$m\in M$.

\section{Rank of an equivariant torsion--free coherent sheaf}

In this section, we derive a formula for the rank of an equivariant torsion--free coherent sheaf on $X$. 

\begin{defn} \label{genf} For an equivariant  torsion--free coherent sheaf $\mc{F}$  and an $n$--dimensional cone $\sigma$, 
define 
$$Gen(\widehat{F}^{\sigma})\,=\, \{m' \,\in\, M\,\mid\, \dim F^{\sigma}_{m} \, <\, \dim F^{\sigma}_{m'} \ \ \forall\ m\, 
<_{\sigma} \,m' \} \,.$$
\end{defn}
  
Since $\mc{F}$ is a coherent sheaf, it follows that $Gen(\widehat{F}^{\sigma}) $ is finite for every $\sigma$.
Note that the finite collection of graded vector spaces
$$\{F^{\sigma}_m\,\mid\, m \,\in\, Gen(\widehat{F}^{\sigma}) ,\ \sigma \,\in\, \Delta(n) \} \,,$$ 
and the isomorphisms $\eta_{\tau \sigma}$ of the previous section, together
determine the $\Delta$--family $\widehat{F}^{\Delta}$.
  
  A coherent sheaf is locally free on some open subset and its rank equals the rank of its restriction to such an open subset.
  By equivariance, the coherent sheaf $\mc{F}$ must be locally free on the dense torus orbit $X_{\{0\}}$, where  $\{0\}$ denotes the trivial cone. 
  By localizing to the dense torus orbit, we find that ${\rm rank}\, \mc{F} \,= \,\dim F^{\{0\}}_m$ for all $m \,\in\, M$.
Now it is straight-forward to check that
\begin{equation}\label{rank1}
{\rm rank}\, \mc{F} \,= \,\dim F^{\sigma}_m \,, \; {\rm where}\; 
m' <_{\sigma} m \; {\rm for \; all} \;  m' \,\in\, Gen(\widehat{F}^{\sigma})\, ,\end{equation}  
for any $\sigma \in \Delta(n)$.

Let $\alpha$ be any one dimensional cone. Note that the spaces $F^{\alpha}_m$ and $F^{\alpha}_{m'}$ are isomorphic if $m-m' \,\in\, S_{\alpha^{\perp}}$,
or in other words if $\langle \alpha,\, m \rangle \,=\, \langle \alpha,\, m' \rangle$.

\begin{defn}\label{de}  
	Let $\mc{F}$ be a $T$--equivariant   torsion--free  coherent sheaf on $X$.  
	For  a subcone $\alpha\,\in\, \Delta(1) $ and $\lambda \,\in\, \mathbb{Z}$, define 
	$$d (\mc{F}, \alpha, \lambda) \,=\, \dim F^{\alpha}_m, \quad {\rm where} \quad \lambda = \langle {\alpha}, m \rangle \,. $$
	Define 
	$$  e(\mc{F}, \alpha, \lambda) = d (\mc{F}, \alpha, \lambda) - d (\mc{F}, \alpha, \lambda -1)\,. $$ 
\end{defn}

We remark that $d (\mc{F}, \alpha, \lambda)\,=\, \dim F^{\alpha}_{[m]}$, where $[m]$ 
denotes the equivalence class of $m$ in $M_{\alpha}$.

\begin{lemma}\label{rank} For an equivariant torsion--free coherent  sheaf $\mc{F}$ on $X$, the equality
	$$ {\rm rank} (\mc{F}) \,=\,  \sum_{\lambda \in \mathbb{Z}}   e(\mc{F}, \alpha, \lambda)$$
	holds for all $\alpha \,\in\, \Delta(1)$. \end{lemma}

\begin{proof} Fix $\alpha \in \Delta(1)$. 
	We may identify $M_{\alpha}$ with $\mathbb{Z}$ via the association $[m]\,\longmapsto\, \langle \alpha,\, m \rangle$.
	Note that $F^{\alpha}_{j}$ has a natural inclusion in $F^{\alpha}_{j+1}$ under the multiplication by the 
	character $\chi^{\alpha}(1)$ of the torus $T_{\alpha}$; see  \eqref{chim2}.
	
	By the finiteness of  $Gen(\widehat{F}^{\sigma})$ for all $\sigma$, the set $$S\,:=\,\{ \lambda \,\in\, \mathbb{Z} \,\mid\,  
	e(\mc{F}, \alpha, \lambda) \,\neq 0\,\}$$ is finite. 
	Suppose $\lambda_1 \,<\,  \cdots \,<\, \lambda_m$ are the elements of $S$. 
	By \eqref{rank1}, we have
	\begin{equation}\label{rank3} {\rm rank}\,\mc{F} \,=\, \dim (F^{\alpha}_j)\,, \; {\rm for} \;  j \ge \lambda_m \,. \end{equation} 
	
	Recall the  affine subvariety $A_{\alpha}$ of $X_{\alpha}$ defined by the ideal generated by $$\{\chi(m)-1\,\mid\, m \,\in\, S_{\alpha}^{\perp}\} \,.$$
	This is, in fact, a nonsingular affine curve. Therefore, as $\mc{F}$ is torsion--free,  $\Gamma(A_{\alpha}, \mc{F})$ is a free
	$k[A_{\alpha}]$--module.
	Let $\{f_i \,\in\, F^{\alpha}_{\lambda_i} \,\mid\, 1\,\le\, i \,\le \,m \}$ be any collection such that 
	$f_i$ is not in the image of $F^{\alpha}_{\lambda_i -1}$ under the multiplication by $\chi^{\alpha}(1)$. 
	To prove the lemma, in view \eqref{rank3} of it is enough to show that $f_1,\, \cdots,\, f_m$ are $k[A_{\alpha}]$--linearly independent. 
	
	Assume the contrary, namely, $f_1,\, \cdots,\, f_m$ are not $k[A_{\alpha}]$--linearly independent. Note that $k[A_{\alpha}] $ is generated as a $k$-vector space
	by $\{ \chi^{\alpha}(p)\,\mid\, p \,\in\, \mathbb{Z}_{\ge 0}  \}$. Therefore, we have 
	\begin{equation}\label{relation} \sum_{i=1}^m  c_i \, \chi^{\alpha}(d_i) \, f_i =0 \end{equation}
	for some nontrivial $c_i \in k$, and some nonnegative integers $d_i$. We may assume without loss of generality that $c_m \neq 0$. 
	Moreover, by considering the direct sum decomposition \eqref{isod},   we may assume without loss of generality that each 
	summand $\chi^{\alpha}(d_i) \, f_i$ belongs to the same graded component $F^{\alpha}_d$, where $d \ge m$.
	Then, dividing \eqref{relation} by $c_m\, \chi^{\alpha}(d-m)$, we obtain that 
	$f_m$ belongs to the image of $F^{\alpha}_{\lambda_{m-1}}$ which is a contradiction.
	This concludes the proof. 
\end{proof}

\section{Equivariant coherent subsheaves of $TX$}

Let $\sigma$ be an $n$-dimensional cone in $\Delta$. Let $\alpha^{\sigma}_1, \,\cdots,\, \alpha^{\sigma}_n$ 
be the primitive integral generators of the one-dimensional
faces of $\sigma$. Since $X_{\sigma}$ is nonsingular, the vectors $\alpha^{\sigma}_1,\, \cdots,\, \alpha^{\sigma}_n$ form a basis of the 
$\mathbb{Z}$--module $N$.
Let
$${\sigma}^{\vee} \,=\, \{m \,\in\, M\otimes \mathbb{R}\,\mid\,
\langle m,\, v \rangle \,\geq\, 0 \ \ \forall \ v \,\in\, \sigma \}$$ be the dual
cone of $\sigma$. Define $ m_i^{\sigma} \,\in\, \sigma^{\vee} \bigcap M$ by 
\begin{equation} \label{dual} 
\langle m_i^{\sigma}, \alpha^{\sigma}_j \rangle = \delta_{ij} \, , 
\end{equation} 
where $\delta_{ij}$ denotes the Kronecker delta. Note that $m_1^{\sigma},\, \cdots,\, m_n^{\sigma}$ form a $\mathbb{Z}$--basis of $M$.

Set $\mc{E} = TX$. Then $ E^{\sigma} $ is a free $\mc{O}_{X_{\sigma}}$--module of rank $n$, with generators having $T$--weights
$-m^{\sigma}_i\,, 1 \le i \le n$. 
To be precise, let
$$z_i = \chi(m^{\sigma}_i) \quad {\rm  and } \quad \partial_{z_i} = \frac{\partial}{\partial z_i} \quad {\rm for} \; 1 \le i \le n \,.$$ 
Then $ E^{\sigma} $ is freely generated over $\mc{O}_{X_{\sigma}}$ by $\{\partial_{z_1},\, \cdots,\, \partial_{z_n}\}$. In the convention followed by Perling \cite{Per}, $T$ acts on $\chi(m)$ with weight $m$. Hence, the section $\partial_{z_i}$ has $T$--weight $- m^{\sigma}_i$.
Note that $\dim E^{\sigma}_{-m^{\sigma}_i} =1$ for $1 \le i \le n$.
Consequently, the $\sigma$--family $\widehat{E}^{\sigma}$ has the following properties:
\begin{equation}\label{char1}
\begin{array}{l}
(a)\; \dim E^{\sigma}_{m} \,=\, | \{ m_i^{\sigma}\,\mid\, - m_i^{\sigma} \,\leq_{\sigma}\, m\} |\, . \\
(b)\; {\rm For \; every}\:  m \,\leq_{\sigma}\, m' , \; {\rm the \; multiplication  \; map \; } \chi^{\sigma} (m,m')\;  
{\rm  is\;  injective}.
\end{array}
\end{equation}

For a subcone $\tau \,<\, \sigma$, the $\tau$--family $\widehat{E}^{\tau}$ satisfies the following:
\begin{equation}\label{char2}
\begin{array}{l}
(a) \;  \dim E^{\tau}_{m} \,=\, | \{ m_i^{\sigma} \,\mid\,  - m_i^{\sigma} \,\leq_{\tau}\, m  \} |\, ,  \\
(b) \;  {\rm For \; every}\;   m \leq_{\tau} m',   \; {\rm the \; multiplication  \; map \; } \chi^{\tau} (m,m')\;  
{\rm  is\;  injective}.
\end{array}
\end{equation}
 By \eqref{char1}, 
  \begin{equation}\label{etx}
  e(TX, \alpha, \lambda) = \left\{ \begin{array}{ll}  
  1 & {\rm if} \;    \lambda= - 1,  \\
  n-1 & {\rm if} \;    \lambda=0,  \\    
  0  & {\rm otherwise},
  \end{array} \right. \end{equation}
  for every $\alpha \in \Delta(1)$.
  
  Let $\mc{F}$ be any equivariant coherent subsheaf of $TX$. Then, $\mc{F}$ is torsion--free.
  Therefore, if $ e(\mc{F}, \alpha, \lambda) \,\neq\, 0$, then 
  $ d(\mc{F}, \alpha, \lambda) \,\neq\, 0$, and consequently
  $E^{\alpha}_m \,\neq\, {0}$, where $\langle \alpha, \,m \rangle \,=\, \lambda$ and $E^{\alpha} \,=\, \Gamma(TX_{\alpha})$. 
  As a result, we have
  \begin{equation}\label{lamb}  e(\mc{F}, \alpha, \lambda) \,\neq\, 0 \ \implies\  \lambda \,\ge\, -1 \,.\end{equation} 
  Hence, for an equivariant coherent subsheaf $\mathcal{F}$ of $TX$, the rank formula of Lemma \ref{rank} takes the form,
  \begin{equation}\label{rank2}
  {\rm rank} (\mc{F}) \,=\,  \sum_{\lambda \in \mathbb{Z}_{\ge -1}}   e(\mc{F}, \alpha, \lambda) 
  \end{equation}  
  where $\alpha$ is an arbitrary element of $\Delta(1)$.
  
  Let $D_{\alpha}$ denote the torus invariant Weil divisor of $X$ corresponding to $\alpha \in [\Delta(1)]$. 
  
  \begin{theorem}[{\cite[Corollary 1.2.18]{Kool1}}]\label{thmc1}
  	The first Chern class of $\mc{F}$ has the expression 
  	$$c_1(\mc{F}) \,=\, - \sum_{\alpha\in \Delta(1), \lambda \in \mathbb{Z}} \lambda \, e(\mc{F}, \alpha, \lambda) \, D_{\alpha} \,. $$
  \end{theorem}
  
  Let $H = \sum a_{\alpha}D_{\alpha}$ be a polarization of $X$; in other words, $H$ is an  ample Cartier divisor on $X$.  Let $P$ be the polytope of $X$ in $M \otimes \mathbb{R}$ associated to $H$.
  The convex polytope $P$ basically encodes the linear system of $H$. It has a facet $P_{\alpha}$ corresponding to each $\alpha$. 
  The facet $P_{\alpha}$ lies in the hyperplane $\{v \,\in\, M \otimes \mathbb{R}\,\mid\, \langle v,\, \alpha \rangle \,=\, -a_{\alpha}\}$.
  Now $P$ is easily determined from these supporting hyperplanes, and has the explicit formula $$P \,=\,  \{v \,\in\, M \otimes \mathbb{R}
  \,\mid\, \langle v,\, \alpha \rangle \,\ge\, -a_{\alpha} \ \ \forall \ \alpha \,\in\, [\Delta(1)] \} \,.$$ 
  
  By  \cite[Corollary on p. 112]{Ful}, the intersection product 
  \begin{equation}\label{int} D_{\alpha} \cdot H^{n-1} \,= \,(n-1)! \, {\rm Vol}(P_{\alpha} ) \,, 
  \end{equation}
  where the volume is measured with respect to the lattice $S_{\alpha^{\perp}}$ (see \eqref{ssperp}).

Using Theorem \ref{thmc1}, \eqref{int}, \eqref{etx} and \eqref{lamb}, we now have
\begin{equation}\label{degtx} \deg\, TX =  (n-1)! \, \sum_{\alpha \in \Delta(1)} {\rm Vol}(P_{\alpha} ) \,, 
\end{equation}
and
\begin{equation}\label{degf} \deg \mc{F} =  - (n-1)! \,   \sum_{\alpha\in \Delta(1), \lambda \in \mathbb{Z}_{\ge-1}} \lambda \, e(\mc{F}, \alpha, \lambda) {\rm Vol}(P_{\alpha} ) \,. 
\end{equation}

\begin{lemma} An upper bound for the slope of an equivariant coherent subsheaf $\mathcal{F}$ of $TX$ of rank $r$ is given by
$$ \mu( \mc{F})  \le  \frac{(n-1)!}{r} \,   \sum_{\alpha\in \Delta(1)}  {\rm Vol}(P_{\alpha} ) \,.  $$
\end{lemma}

\begin{proof} Fix an $\alpha \in \Delta(1)$ and an $n$-dimensional cone $\sigma$ containing $\alpha$. 
Consider the summand
 $$   - (n-1)! \,{\rm Vol}(P_{\alpha} )\, \sum_{ \lambda \in \mathbb{Z}_{\ge-1}} \lambda \, e(\mc{F}, \alpha, \lambda) \, 
$$
of \eqref{degf} corresponding to $\alpha$.  As the sum is decreasing in the $\lambda$'s, the optimal choice would be
$\lambda \,=\, -1$. However if $ e(\mc{F}, \alpha, -1) \,\neq\, 0$, then there is an element of the form 
$g(z)\partial{z_{\alpha}}$ in $F^{\sigma}$ with $g(z) \,\in\, k[S_{\sigma}]$. Moreover all such elements are multiples of one another 
when restricted to the dense torus $X_{\{0\}}$. Therefore, we have $ e(\mc{F}, \alpha, -1)  \,\le\, 1$. Hence, 
 $$   - (n-1)! \,{\rm Vol}(P_{\alpha} )\, \sum_{ \lambda \in \mathbb{Z}_{\ge-1}} \lambda \,
e(\mc{F}, \alpha, \lambda) \,\le\,  (n-1)! \,{\rm Vol}(P_{\alpha} )\,.$$
This completes the proof of the lemma.
\end{proof}

The above upper bound is not very sharp. We will use finer estimates in what follows.

\begin{lemma}\label{compare} Suppose  $\mc{F} ,\; \mc{G}$  are equivariant coherent subsheaves of $TX$ having the same
 rank. If $\mc{F}$ is a proper subsheaf
of $\mc{G}$, then $\mu(\mc{F}) <  \mu(\mc{G})$.
\end{lemma}

\begin{proof}
For subsheaves  $\mc{F} \,\subseteq\, \mc{G}$  of $TX$, we have $F^{\alpha}_m \,\subseteq\, G^{\alpha}_m$.
 If, in addition, ${\rm rank}(\mc{F})\,=\, {\rm rank}(\mc{G})$, 
then  Lemma \ref{rank} implies that generators of $\mc{F}$ are associated with bigger values of $\lambda$.
Therefore, we have
$$   \sum_{\alpha\in \Delta(1), \lambda \in \mathbb{Z}} \lambda \, e(\mc{G}, \alpha, \lambda) \, {\rm Vol}(P_{\alpha})
\le \sum_{\rho\in \Delta(1), \lambda \in \mathbb{Z}} \lambda \, e(\mc{F}, \alpha, \lambda) \, {\rm Vol}(P_{\alpha}) \,.$$

Thus, when $\mc{F} \,\subseteq\, \mc{G} \,\subseteq\, TX$, and  ${\rm rank}(\mc{F}) \,=\, {\rm rank}(\mc{G})$, 
by  \eqref{degf},  we have 
\begin{equation}\label{degrank}
 \deg \,\mc{F} \,\le\, \deg \,\mc{G} \,.
 \end{equation}

Now, suppose $\mc{F}$ is a proper subsheaf of $\mc{G}$. Then there exists an $n$-dimensional cone $\sigma$ and an $m \,\in\,
 Gen(\widehat{G}^{\sigma})$ such that $F^{\sigma}_m \,\subsetneq\, G^{\sigma}_m$. Then  there exist at least one
$\alpha_i \,\in\, \sigma \cap \Delta(1)$ such that 
$e(\mc{F}, \alpha_i, \langle \alpha_i,m \rangle) < e(\mc{G}, \alpha_i, \langle \alpha_i,m \rangle)$. The lemma follows.
\end{proof}

\section{Hirzebruch surfaces}

In this section we study semi-stability of tangent bundle for smooth projective toric surface. Following lemma is very crucial in computing degree of rank $1$ subsheaves of the tangent bundle. 
\begin{lemma}\label{collection}  
To any equivariant rank one coherent subsheaf $\mc{F}$ of $TX$ on a complete nonsingular toric surface $X$, one can associate an integral 
vector \begin{equation}\label{coll1}
 \vec{\lambda}\,:=\,( \lambda_1, \,\cdots, \,\lambda_p)
 \end{equation}
  where  \begin{enumerate}
 \item $p=  |\Delta(1)|$,  and $\Delta(1) = \{ \alpha_1, \ldots, \alpha_p \}$
 \item $ e(\mc{F} , \alpha_i, \lambda_i) =1 $ for each $i$,
  \item each $ \lambda_i \in \mathbb{Z}_{\ge -1}$, and
 \item $ (\lambda_i, \lambda_j) \neq (-1,-1)$   if
 $ \alpha_i, \alpha_j $  form  a  cone in $\Delta$. 
 \end{enumerate}
  \end{lemma}

\begin{proof} Given a rank one subsheaf $\mc{F}$ of $TX$, for each $\alpha_i \in \Delta(1)$, by  \eqref{rank2} and  \eqref{lamb}
there exists a unique
  $\lambda_i \in \mathbb{Z}_{\ge -1}$ such that $e(\mc{F} , \alpha_i, \lambda_i) =1 $.
  
  Now, suppose    $\alpha_i, \alpha_j$
  generate a cone $\sigma$.  Denote the characters $\chi(m^{\sigma}_i), \; \chi(m^{\sigma}_j)$ by $z_i$ and $z_j$ respectively. 
  Then $\lambda_i = -1$ implies that $F^{\sigma}_m \neq 0$ for some $m = -m^{\sigma}_i + k m^{\sigma}_j$ where $k \ge 0$. Therefore, 
  $F^{\sigma}$    has a generator of the form $z_j^k \partial_{z_i}$.  Similarly, 
  $F^{\sigma}$ has a generator of the form $z_i^l \partial_{z_j}$ if $\lambda_j =-1$. Thus rank$(\mc{F})$ would exceed one if 
  $(\lambda_i, \lambda_j)= (-1,-1)$. 
    \end{proof}

However, not every vector of type \eqref{coll1} corresponds to an equivariant rank one coherent subsheaf of $TX$.
 To illustrate this, we consider the example of a Hirzebruch surface.  The Hirzebruch surface $X= \mathbb{F}_m$ is 
 a projective toric surface which may be obtained by  the projectivization of the bundle $\mc{O}_{\mathbb{P}^{1}} \oplus
\mc{O}_{\mathbb{P}^{1}} (-m) $ on $\mathbb{P}^{1}$. 
    $\mathbb{F}_m$ has fan $\Delta$ with $\Delta(1)\,=\, \{ \alpha_1,\, \cdots, \,\alpha_4\}$ where
  $$ [ \alpha_1 \; \alpha_2 \; \alpha_3 \; \alpha_4 ] = \left[ \begin{array}{rrrr}  1 & 0 & -1 & 0 \\ 0 & 1 & m & -1 \end{array} \right] \,.$$
We denote the $2$-dimensional cone generated by $\alpha_i$ and  $ \alpha_{i+1({\rm mod }~4)} $ by $\sigma_i$.

Assume that $m \,\ge\, 1$. Consider a collection \eqref{coll1} such that $ \; \lambda_1 = \lambda_3 = -1 \,.$
Denote $$t_1= \chi(m^{\sigma_1}_1)= \chi((1,0)) , \quad t_2 = \chi(m^{\sigma_1}_2)=\chi((0,1)) \,.$$
\noindent 
As $\lambda_1 = -1$, there is an element $s_1$ of the form  $ t_2^k \partial_{t_1}$ in $F^{\sigma_1}$. Note that $m^{\sigma_2}_2 = (m,1) $ and 
$m^{\sigma_2}_3= (-1,0)$. (Here, it is convenient to continue to use the notation introduced in the proof of Lemma \ref{collection} as opposed to the last section.) Denote $$z_2= \chi(m^{\sigma_2}_2)=  t_1^m t_2, \quad z_3 = \chi(m^{\sigma_2}_3)= t_1^{-1}\,.$$
This implies that
 $$ t_1 = z_3^{-1}, \quad t_2 =   z_2 z_3^m  \,.$$ Therefore, we have the Jacobian $$\frac{\partial (t_1, t_2)}{ \partial (z_2, z_3)} =
  \left[ \begin{array}{rr}  0 & -z_3^{-2}  \\ z_3^m  & mz_2 z_3^{m-1}  \end{array} \right] = 
   \left[ \begin{array}{rr}  0 & -t_1^{2}  \\  t_1^{-m}  & m t_1 t_2  \end{array} \right]   \,.$$ 
   Thus we have 
$$  {\partial_{z_2} } =  t_1^{-m} \partial_{t_2}, \quad  {\partial_{z_3} } =  -t_1^{2}  \partial_{t_1} + m t_1 t_2 \partial_{t_2} \,.  $$  
Again, as $\lambda_3= -1 $, there must be an element $s_2$ in $F^{\sigma_2}$ of the form 
$ z_2^l \partial_{z_3} $.
The restrictions of $s_1$ and $s_2$ to the dense torus $X_{\{0\}}$ are independent. Therefore, $\mc{F}$ must have rank $2$, which is a contradiction. 
  
However, it can be verified that the vector  
 \begin{equation}\label{coll3} \vec{\lambda}  = (0,-1,0,-1) \, \end{equation}
corresponds to a rank one coherent subsheaf of $TX$ on every $\mathbb{F}_m$. An analogous, but more general
example is worked out in Theorem \ref{family}.

 \begin{theorem}\label{HS} The tangent bundle of the Hirzebruch surface $\mathbb{F}_m$ is unstable with respect to every polarization if $m\ge 2$. \end{theorem}
 
 \begin{proof}  Let $D_i$ denote the divisor of $X= \mathbb{F}_m$ corresponding to $\alpha_i$, $1\le i\le 4$. The $D_i$'s generate the Picard group of $X$. 
  It is easy to check that  $H = \sum a_i D_i $ is ample if and only if $a := a_1 + a_3 - m a_2 > 0$ and $b := a_2 + a_4 > 0$. The polytope $P$ associated to $H$ has vertices $A=(-a_1, -a_2)$, $ B=(a_3-ma_2, -a_2)$,   $C=(ma_4+a_3,a_4)$ and
   $D=(-a_1,a_4)$. The facets (edges) are $P_{\alpha_1}=AD $, $P_{\alpha_2} = AB$, $P_{\alpha_3}= BC$ and $P_{\alpha_4}= CD$. Their
   volumes are $b$, $a$, $b$ and $a+mb$ respectively. 
   
   The slope of a rank one coherent subsheaf $\mc{F} $ associated to the collection \eqref{coll1} is 
\begin{equation}\label{slopecoll}
 \mu(\mc{F}) =  \deg(\mc{F})=   - \,   \sum_{\alpha_i \in \Delta(1)}  \, \lambda_i {\rm Vol}(P_{\alpha_i} ) \,. 
 \end{equation}
 It follows that the slope is a decreasing function of each $\lambda_i$.   
   Since there is no rank one equivariant coherent subsheaves of $TX$ with $\lambda_1 =-1 = \lambda_3$,  the 
  rank one equivariant  coherent subsheaf with maximum slope corresponds to the collection \eqref{coll3}. The slope of this 
  subsheaf is
   $$ \mu(\mc{F}) \,=\, {\rm Vol}(P_{\alpha_2} ) + {\rm Vol}(P_{\alpha_4} ) \,=\, 2a +mb  \,. $$ 
      Note that the rank one subsheaves with $\vec{ \lambda} =( -1,0,0,0)  $, $(0,-1,0,0)$, $(0,0,-1,0)$  and $(0,0,0,-1)$, if they exist, have slopes $b$, $a$, $b$ and $a+mb$ rspectively; and all of them are less than $2a +mb$. 
   
  On the other hand, by \eqref{degtx},  the slope of $TX$ is 
  $$ \mu(TX) = \frac{1}{2} \sum_{i=1}^4 {\rm Vol}(P_{\alpha_i} ) =  a + \frac{m+2}{2} b \,.$$
  Since $a$ and $b$ are positive, the theorem follows.
   \end{proof}
   
   It is known that $T\mathbb{F}_0$ and $T\mathbb{F}_1$  are semistable with respect to the anti-canonical polarization, namely when 
   each $a_i =1$. We can deduce more about $\mathbb{F}_1$ from the above calculations. 
   
\begin{corollary}\label{HS2} For a polarization  with $ 2a <  b$, for example when $ 0< 2a_1 < a_4$ and $a_2=a_3= 0$, 
the tangent bundle $T\mathbb{F}_1$ is stable. On the other hand, if $2 a > b$, for example when $0< a_4 < 2a_1$ and $a_2=a_3= 0$, then
$T\mathbb{F}_1$ is unstable.
\end{corollary}
   
\section{Projective spaces}    
 
 The method of the previous section can be adapted to give an alternative proof of the well-known result that the tangent bundle of the complex projective space $\mathbb{P}^n$ is
 stable with respect to the anti-canonical polarization (cf. \cite[Section 1.4]{HL}).
  
\begin{theorem} \label{projective spaces}  $T\mathbb{P}^n$ is stable with respect to the anti-canonical polarization for every $n$.
\end{theorem}

\begin{proof} Let $\Delta$ be the fan of $\mathbb{P}^n$.  Then $\Delta(1) \,=\, \{\alpha_1, \,\cdots,\, \alpha_n,\, \alpha_{n+1}\}$ where
$$\{\alpha_1\,=\, (1,\, 0,\,\cdots,\, 0,\,0), \,\cdots,\, \alpha_n\,=\, (0,\,0,\,\cdots,\,0,\,1)\}$$ is the standard basis of $\mathbb{R}^n$ and 
 $$ \alpha_{n+1} \,=\, -\sum_i \alpha_i\,= \,(-1,\,-1,\, \cdots,\,-1)\,.$$
 $\Delta$ has $n+1$ cones of dimension $n$ which may be enumerated as 
 $$\sigma_i\,=\, \langle \alpha_1,\, \cdots ,\, \widehat{\alpha}_i,\, 
 \cdots,\, \alpha_{n+1} \rangle, \ \ 1 \,\le\, i \,\le\,  n+1 \,.$$
 Let $P$ denote the polytope of $\mathbb{P}^n$ with respect to the anti-canonical polarization.  For any pair $(i,j)$ there is an automorphism 
 $A_{i j}$ of the lattice $N$ that interchanges  $\alpha_i$ and $\alpha_j$ keeping the other $\alpha$'s fixed. This implies that  
  every facet $P_{\alpha_i}$ of $P$ has the same volume, say $V_n$.  Hence,  by \eqref{degf}, the slope of an equivariant rank $r$ coherent subsheaf $\mc{F}$ of the tangent bundle is 
\begin{equation}\label{muf} \mu(\mc{F}) =   - (n-1)! \, \frac{V_n}{r}  \sum_{\alpha\in \Delta(1), \lambda \in \mathbb{Z}_{\ge-1}} \lambda \, e(\mc{F}, \alpha, \lambda)  \,. 
\end{equation}

We claim that, if $r< n$, then there can be at most $r$ many $\alpha_i$'s with $e(\mc{F}, \alpha_i, -1) = 1 $. This would imply that $\mu(\mc{F}) \le 
  (n-1)! \, V_n $. 
However, by \eqref{degtx}
 $$\mu(T\mathbb{P}^n) = (n-1)! \frac{(n+1) V_n}{n} >   (n-1)! \, V_n  \,.$$  
 Thus $\mu({\mc{F}}) < \mu(T\mathbb{P}^n)$ for every proper sub sheaf $\mc{F}$ of $T\mathbb{P}^n$. 
 
To prove the claim, assume that there are at least $(r+1)$ many $\alpha_i$'s with $e(\mc{F}, \alpha_i, -1)\,=\, 1$. Since $r+1 \le n$, there exists 
an $n$-dimensional cone $\sigma$ containing $(r+1)$ of these $\alpha_i$'s. Note that the corresponding $(r+1)$ many $\partial_{z^{\sigma}_i}$'s (up 
to multiplication by characters) are all generators of $F^{\sigma}$, contradicting the fact that ${\rm rank}(\mc{F})\,=\, r$.
\end{proof}

\section{Unstable Fano examples in higher dimensions}  

If $X$ is a Fano toric variety, then the polytope $P$ corresponding to the anti-canonical polarization  is a reflexive polytope (cf. \cite{Bat}). This implies that $P$ has integral vertices and the origin is the unique integral point which is in the interior of $P$.
When $X$ is a surface, it is easy to tabulate such polytopes. One easily checks that the tangent bundle of a nonsingular Fano toric surface is semistable with respect to the anti--canonical  polarization. However, it is known that the tangent bundle of any nonsingular Fano surface is semistable with respect to the anti--canonical polarization (cf. \cite{ Fah}). 
 
The product of two Fano varieties $X_1 \times X_2 $ is  Fano. If $TX_1$ and $TX_2$ are both semistable with respect to their
respective  anti--canonical polarizations, then the same holds for $T(X_1 \times X_2)$ (cf. \cite{Stef}). This yields more examples of semistable 
Fano toric varieties. However there are a lot of Fano toric varieties with unstable tangent bundle in higher dimensions, as the following result shows. For $n=3$, the result below had appeared in \cite{Stef}.

\begin{theorem}\label{family} Suppose $X$ is the Fano toric variety  
$\mathbb{P}( \mc{O}_{\mathbb{P}^{n-1}} \oplus \mc{O}_{ \mathbb{P}^{n-1} }(m) )$ where $0 < m \le n-1$ and $n \ge 3$. Then $TX$ is unstable with respect to the anti-canonical polarization.  
 \end{theorem}
 
 \begin{proof} Let $\Delta$ be the fan of $X$.  Then $\Delta(1) \,=\, \{\alpha_1,\, \cdots,\, \alpha_n, \,\alpha_{n+1},\, \alpha_{n+2}\}$, where
  $\{\alpha_1\,=\, (1,\, 0,\,\cdots,\, 0,\,0),\, \cdots, \,\alpha_n\,=\, (0,\,0,\,\cdots,\,0,\,1)\}$ is the standard basis of $\mathbb{R}^n$ and 
 $$ \alpha_{n+1} = - \alpha_n, \quad  \alpha_{n+2}= -\sum_{j=1}^n \alpha_j + (m+1)\alpha_n \,=\, (-1,\,-1,\, \cdots,\,-1,\, m)\,.$$
 $\Delta$ has $2n$ cones of dimension $n$. They may be divided into two groups, upper and lower. The upper cones contain $\alpha_n$
but not $\alpha_{n+1}$, and vice versa for the lower cones. There are $n$ upper cones, each missing one of the remaining $\alpha_i$'s, 
and similarly for lower cones.

We verify that there is an equivariant subsheaf $\mc{F}$ of $TX$ of rank $1$ associated to the vector
$$\vec{\lambda} \,=\, (0,\, \cdots,\, 0,\, -1,\,-1)  \, ,$$
i.e. $ \lambda_n = \lambda_{n+1} = -1, \lambda_j= 0 \; {\rm  for\; other \,} j $ (see Lemma \ref{collection}).
Note that for this $\vec{\lambda}$, in every $n$--dimensional cone $\sigma$ there is exactly one $\alpha_i$,  say $\alpha_{i_0}$,  such  that $e(\mc{F}, \alpha_i, -1) \neq 0$. The number $i_0$ is $n$ for upper cones, and $n+1$  for lower cones.

Consider the upper cone  $\sigma$ generated by $\alpha_1,\, \cdots,\, \alpha_n$. Denote $\chi(m^{\sigma}_i)$ by $t_i$ for $1\,\le\, i\,\le\, n$.
Then $\lambda_n = - 1$ implies that $F^{\sigma}$ is generated by an element of the form $g(t_1,\, \cdots,\,t_{n-1})\partial_{t_n}$. 
Set $g(t_1,\, \cdots,\, t_{n-1}) \,=\,1$.

Now consider any other upper cone  $\tau$. Assume $\tau$ is missing the ray $\alpha_j$, where $j \le n$.  
Let $z^{\tau}_i := \chi(m^{\tau}_i)$ where $1\le i \le n+2$ and $i \neq j, n+1 $.
Then  $F^{\tau}$ is generated by an element of the form $f(z) \partial_{z^{\tau}_n}$. 
We have 
$$z^{\tau}_i = \left\{  \begin{array}{ll} {t_i }{t_j}^{-1} & {\rm if }\;   i \neq n, n+2 \,, \\ 
t_j^m t_n & {\rm if }\; i=n \,, \\
{t_j^{-1}} & {\rm if }\; i=n + 2 \,.
\end{array} \right. $$
It follows easily that $ \partial_{t_n} = t_j^m \partial_{z^{\tau}_n} = (z^{\tau}_{n+2})^{-m}\partial_{z^{\tau}_n} $. 
Note that $\lambda_n =-1$ means $F^{\tau}$ should be generated by an 
element of the form $h(z^{\tau}_1, \,\cdots,\, \widehat{z^{\tau}_n},\, \cdots,\, z^{\tau}_{n+2} ) \partial_{z^{\tau}_n}$.
Set $h \,=\,(z^{\tau}_{n+2})^{-m}$.

The calculations for the lower cones are quite analogous. Finally, it is enough to consider the transition between the upper cone 
$\sigma$ and the lower cone $\delta$ generated by $\alpha_1, \,\cdots, \,\alpha_{n-1}, \,\alpha_{n+1}$. The coordinates on $X_{\delta}$
are $t_1, \,\cdots,\, t_{n-1}$ and $w\,=\,\chi(m^{\delta}_{n+1})\,= \,t_n^{-1}$. As $\lambda_{n+1}\,=\, -1$, we assume that $F^{\delta}$ is generated by 
$\partial_w$. Note that $\partial_w = -t_n^2 \partial_{t_n}$ and $t_n$ is invertible on $X_{\sigma} \cap X_{\delta}$. This confirms the existence of the desired rank one sheaf $\mc{F}$.
 
Consider the anti--canonical divisor $H$ on $X$. The associated polytope is combinatorially equivalent to a prism $S^{n-1} \times I $ 
where $S^{n-1}, \, I$ denote a simplex of dimension $n-1$ and an interval respectively. The top and bottom facets, $P_{\alpha_{n+1}}$ and $P_{\alpha_n}$, are 
$(n-1)$--dimensional simplices and lie on the hyperplanes $x_n=1$ and $x_n=-1$ respectively.   The top facet $P_{\alpha_{n+1}}$ has the vertices 
$$\begin{array} {l} (-1, -1, \cdots, -1,1), \; (n+m-1, -1, \cdots, -1,1), \;  (-1, n+m-1, \cdots, -1,1),   \\ 
\cdots,  (-1, -1, \cdots, n+m -1,1) \,.\end{array}$$
The bottom facet $P_{\alpha_n}$ has the vertices 
$$\begin{array} {l} (-1, -1, \cdots, -1,-1), \; (n-m-1, -1, \cdots, -1,-1), \;  (-1, n-m-1, \cdots, -1,-1),   \\ 
\cdots,  (-1, -1, \cdots, n-m -1,-1) \,.\end{array}$$
We have
$${\rm Vol}(P_{\alpha_{n+1}}) = \frac{(n+m)^{n-1}}{(n-1)!} \quad {\rm and} \quad  {\rm Vol}(P_{\alpha_n}) = \frac{(n-m)^{n-1}}{(n-1)!} \,.  $$
Consequently, using \eqref{degf}, we have, 
$$ \mu(\mc{F}) = (n+m)^{n-1} + (n-m)^{n-1} \,. $$

The prism $P$ has $n$ side facets corresponding to the rays $\alpha_1, \,\cdots,\, \alpha_{n-1}$, and $ \alpha_{n+2}$. Each side facet, in turn, is a prism of height $2$ bounded 
by a facet of $P_{\alpha_{n+1}}$ and $P_{\alpha_n}$ on top and bottom respectively. Each of these side facets have volume 
$$ \frac{ (n+m)^{n-2} + (n-m)^{n-2}}{(n-2)!}\,.$$
Therefore, using \eqref{degtx}, we have, 
$$ \mu(TX) = \frac{(n+m)^{n-1} + (n-m)^{n-1} + n(n-1) ((n+m)^{n-2} + (n-m)^{n-2}) }{n} \,. $$
It is now easy to verify  $\mu(\mc{F}) > \mu(TX)$ using  $(a^{n-2} -b^{n-2}) (a-b) > 0$ where 
$a= n+m$ and $b= n-m$. Here, we have used $n \ge 3$.
The theorem follows.
\end{proof}
 
\section{Fano toric $4$--folds with small Picard number}
 
Steffens has studied the stability of the tangent bundle of all Fano $3$-folds in \cite{Stef}. Moreover, Nakagawa \cite{Nak1, Nak2} has identified 
all Fano toric $4$-folds that admit  Einstein-K\"ahler metrics. We study which Fano toric $4$-folds
with Picard number $\leq 2$  have semi-stable tangent bundle.

A list of Fano toric $4$--folds is given by Batyrev in \cite{Bat2},
see also earlier work of Kleinschmidt \cite{Klei}. There are 10 classes of these varieties with Picard number $\leq 2$, which are:
\begin{multicols}{2}	
\begin{enumerate}
	\item $\mathbb{P}^{4}$
	\item $B_1 = \mathbb{P}(\mc{O}_{\mathbb{P}^{3}} \oplus \mc{O}_{\mathbb{P}^{3}}(3)) $
	\item $B_2 = \mathbb{P}(\mc{O}_{\mathbb{P}^{3}}\oplus \mc{O}_{\mathbb{P}^{3}}(2)) $ 
	\item $B_3 = \mathbb{P}(\mc{O}_{\mathbb{P}^{3}}\oplus \mc{O}_{\mathbb{P}^{3}}(1)) $
	\item $B_4 = \mathbb{P}^{1} \times \mathbb{P}^{3}$
	\item $B_5 = \mathbb{P}(\mc{O}_{\mathbb{P}^{1}}^{3} \oplus \mc{O}_{\mathbb{P}^{1}}(1)) $
	\item $C_1 = \mathbb{P}(\mc{O}_{\mathbb{P}^{2}}^{2} \oplus \mc{O}_{\mathbb{P}^{2}}(2)) $
	\item $C_2 = \mathbb{P}(\mc{O}_{\mathbb{P}^{2}}^{2} \oplus \mc{O}_{\mathbb{P}^{2}}(1)) $
	\item $C_3 = \mathbb{P}(\mc{O}_{\mathbb{P}^{2}} \oplus \mc{O}_{\mathbb{P}^{2}}^{2}(1)) $
	\item $C_4 = \mathbb{P}^{2} \times \mathbb{P}^{2}$
\end{enumerate}
 \end{multicols}
 However, we will first need to fine-tune our tools to deal with subsheaves of higher rank. We will describe a 
generalization of Lemma \ref{collection} for rank $r$ subsheaves. 
 
 \begin{lemma}\label{matrix} To any equivariant rank $r$ coherent subsheaf $\mc{F}$ of $TX$ on a complete nonsingular toric variety $X$,
  one can associate a unique  $r \times p$ matrix of integers 
 $ \Lambda_{r \times p} := ( \lambda_{ij}) \, ,  $
  where \begin{enumerate} 
  \item $p=  |\Delta(1)|$, 
   \item $ e(\mc{F} , \alpha_j, \lambda_{ij}) \neq 0 $ for each $1 \leq j \leq p$,  
  \item   $-1 \,\leq\, \lambda_{1j} \,\leq\, \lambda_{2j} \,\leq\, \cdots \,\leq\, \lambda_{rj}$ for every $j$, 
  \item $ \sum  e(\mc{F} , \alpha_j, \lambda_{ij}) = r $ for every $j$, where the sum is over any maximal set of row indices $i$ such that
  corresponding $\lambda_{ij}$'s in the column 
  $j$ are distinct.
  \item $ (\lambda_{ij_1}, \,\cdots,\, \lambda_{ij_{r+1}} ) \,\neq\, (-1,\cdots,-1)$ if
 $ \alpha_{j_1},\, \cdots,\, \alpha_{j_{r+1}}$  form  a  cone in $\Delta$,  
 and
 \item $(\lambda_{1j},  \lambda_{2j}) \neq (-1,-1)$ for any $j$.
 \end{enumerate} \end{lemma}
 
 \begin{proof} We are tabulating which $ e(\mc{F} , \alpha_j, \lambda) \neq 0 $ such that each entry $\lambda_{ij}$
 in the  $j$--th column  contributes $1$ to $e(\mc{F}, \alpha_j, \lambda_{ij})$.   This implies that if $ e(\mc{F} , \alpha_j, \lambda) = k $, then $k$ entries of the $j$--th column have entry $\lambda$.  
 
 Then the proof is almost immediate.  Condition $(3)$ is a choice of order made for the sake of uniqueness. Condition $(4)$ follows from 
 Theorem \ref{rank}. Condition $(5)$ ensures that rank of $\mc{F}$ does not exceed $r$.  Condition $(6)$ follows from the dependence of 
 relevant generators over $k[S_{\{0\}}]$.  It is equivalent to saying $e(\mc{F} , \alpha_j, -1) \leq 1$ for each $j$.   
  \end{proof}
 
 We use the associated matrix $\Lambda$ below to give a proof of the semi-stability of the toric Fano $4$-fold 
  $B_5 = \mathbb{P}(\mc{O}_{\mathbb{P}^{1}}^{3} \oplus \mc{O}_{\mathbb{P}^{1}}(1)) $ from the above list.
  
\begin{theorem}\label{detailed example of 4-fold}
Suppose $X$ is the Fano toric $4$--fold $\mathbb{P}(\mc{O}_{\mathbb{P}^{1}}^{3} \oplus \mc{O}_{\mathbb{P}^{1}}(1))  $. 
Then $TX$ is strictly semistable with respect to the anti-canonical polarization.
\end{theorem}

\begin{proof}
Let $\Delta$ be the fan of $X$.  Then $\Delta(1) = \{ \alpha_1=(1,0,0,0), \alpha_2=(0,1,0,0), \alpha_3=(0,0,1,0), \alpha_4=(-1,-1,-1,0), \alpha_5=(0,0,0,1), \alpha_6=(1,0,0,-1) \} $ (see \cite{Bat2}).\\
We have ${\rm Vol}(P_{\alpha_{1}}) = 8, {\rm Vol}(P_{\alpha_{2}}) = \frac{56}{3}, {\rm Vol}(P_{\alpha_{3}}) = \frac{56}{3}, {\rm Vol}(P_{\alpha_{4}}) = \frac{56}{3}, {\rm Vol}(P_{\alpha_{5}}) = \frac{32}{3}$ and ${\rm Vol}(P_{\alpha_{6}}) = \frac{32}{3}$ with respect to anti-canonical divisor. Then\\
$$ \mu(TX) = \frac{3!}{4} \sum_{i=1}^6 {\rm Vol}(P_{\alpha_i} ) = 128 $$

We denote the $4$-dimensional cones by $\sigma_i$ and we have 8 of them. They are	
\begin{multicols}{2}	
\begin{eqnarray*}
\sigma_1 &=& <\alpha_{2}, \alpha_{3}, \alpha_{4}, \alpha_{6}>\\
\sigma_2 &=& <\alpha_{1}, \alpha_{2}, \alpha_{3}, \alpha_{5}>\\
\sigma_3 &=& <\alpha_{1}, \alpha_{2}, \alpha_{3}, \alpha_{6}>\\	
\sigma_4 &=& <\alpha_{1}, \alpha_{2}, \alpha_{4}, \alpha_{5}>
\end{eqnarray*}

\begin{eqnarray*}
\sigma_5 &=& <\alpha_{1}, \alpha_{2}, \alpha_{4}, \alpha_{6}>\\
\sigma_6 &=& <\alpha_{1}, \alpha_{3}, \alpha_{4}, \alpha_{5}>\\
\sigma_7 &=& <\alpha_{1}, \alpha_{3}, \alpha_{4}, \alpha_{6}>\\	
\sigma_8 &=& <\alpha_{2}, \alpha_{3}, \alpha_{4}, \alpha_{5}>
\end{eqnarray*}
\end{multicols}

Denote
$$t_1= \chi((1,0,0,0)) , \quad t_2 = \chi((0,1,0,0)), \quad t_3 = \chi((0,0,1,0)), \quad t_4 = \chi((0,0,0,1)) \,.$$

First consider rank 1 equivariant subsheaves of $TX$.  There is a possible such subsheaf $\mc{F}$  associated to the 
vector $\vec{\lambda} = ( 0,0,0,0,-1,-1)$.

The other choices for $\vec{\lambda}$ that satisfy the conditions of Lemma \ref{collection} have at most one $-1$ entry. Then it is easy to check
using Lemma \ref{compare} that  $\mc{F}$ has the possible maximum slope among rank 1 subsheaves. The slope of $\mc{F}$ is
$$ \mu(\mc{F}) \,=\, \frac{3!}{1} ({\rm Vol}(P_{\alpha_5} )+ {\rm Vol}(P_{\alpha_6} )) = 128 \,.$$
As this equals  $\mu(TX)$, we need to check the existence of $\mc{F}$.

First consider the $4$--dimensional cones containing $\alpha_5$, these are $\sigma_2, \sigma_4, \sigma_6, \sigma_8$. Define $$x_i := \chi(m^{\sigma_2}_i), 
\quad y_i := \chi(m^{\sigma_4}_i), \quad z_i := \chi(m^{\sigma_6}_i)  \quad {\rm and}  \quad w_i := \chi(m^{\sigma_8}_i)$$
 where $1\leq i \leq 4$.

Then
\begin{equation}\label{cc1}
\begin{array}{llll}
x_1= t_1     & x_2= t_2    & x_3= t_3    &    x_4= t_4 \,, \\
y_1= t_1 t^{-1}_3  &     y_2= t_2 t^{-1}_3   &  y_3= t^{-1}_3    &    y_4= t_4 \,,\\
z_1= t_1 t^{-1}_2 & z_2= t^{-1}_2 t_3  &  z_3= t^{-1}_2  & z_4= t_4 \,, \\
w_1= t^{-1}_1 t_2   &  w_2= t^{-1}_1 t_3  & w_3= t^{-1}_1 &  w_4= t_4 \,. 
\end{array}
\end{equation}

As $\lambda_5 = -1$, there is an element generated by (meaning, a multiple of) $\partial_{x_4}$ in $F^{\sigma_2}$. Similarly, there are elements 
generated by $\partial_{y_4}$, $\partial_{z_4}$ and $\partial_{w_4}$ in $F^{\sigma_4}$, $F^{\sigma_6}$ and $F^{\sigma_8}$, respectively. But, using 
the Jacobians of the transformations between the $t$ and other coordinates, we have 
$$\partial_{x_4}=\partial_{y_4}=\partial_{z_4}=\partial_{w_4}=\partial_{t_4} \,.$$ Thus the generators agree on the dense torus.

Secondly, consider cones containing $\alpha_6$ as a generator, which are $\sigma_1, \sigma_2, \sigma_5, \sigma_7.$ Call now $p_i := \chi(m^{\sigma_1}_i)$, $q_i := \chi(m^{\sigma_3}_i)$, $r_i := \chi(m^{\sigma_5}_i)$ and $s_i := \chi(m^{\sigma_7}_i)$ where $1\le i \le 4$. 
Similar to above computations, find that 
\begin{equation}\label{cc2}
\begin{array}{llll} 
p_1= t^{-1}_1 t_2 t^{-1}_4 & p_2= t^{-1}_1 t_3 t^{-1}_4 & p_3= t^{-1}_1 t^{-1}_4 & p_4= t^{-1}_4 \,,\\
q_1= t_1 t_4                 &   q_2= t_2                  &    q_3= t_3 &      q_4= t^{-1}_4 \,, \\
r_1= t_1 t^{-1}_3 t_4    & r_2= t_2 t^{-1}_3     & r_3= t^{-1}_3    & r_4= t^{-1}_4 \,, \\
s_1= t_1 t^{-1}_2 t_4   & s_2= t^{-1}_2 t_3     & s_3= t^{-1}_2    & s_4= t^{-1}_4 \,.
\end{array}
\end{equation}

As $\lambda_6 = -1$, there are elements generated by $\partial_{p_4}$, $\partial_{q_4}$, $\partial_{r_4}$ and $\partial_{s_4}$ in in 
$F^{\sigma_1}$, $F^{\sigma_3}$, $F^{\sigma_5}$ and $F^{\sigma_7}$ respectively.  Using various Jacobians, we have 
$$\partial_{p_4}=\partial_{q_4}=\partial_{r_4}=\partial_{s_4}= t_1 t_4 \partial_{t_1} - t^{2}_4 \partial_{t_4}\,.$$

However, on the dense torus, the generators $\partial_{t_4}$ and $ t_1 t_4 \partial_{t_1} - t^{2}_4 \partial_{t_4}$ are linearly independent. 
Hence the rank of $\mc{F}$ must be at least $2$, leading to a contradiction. We conclude that $\mu({\mc F}) < 128 = \mu(TX)$ for every 
rank $1$ equivariant subsheaf of $TX$. 

Next we consider rank $2$ equivariant subsheaves of $TX$ with maximum possible slope. By condition $(5)$ of Lemma \ref{matrix}
a subsheaf $\mc{F}$ with $\lambda_{1j} = -1$ for $3$ values of $j$ is only possible when $2$ of the $j$'s are $5$ and $6$. Then, 
using condition $(6)$ of Lemma \ref{matrix}, the slope of 
$\mc{F}$ has the following bound,
$$ \mu(\mc{F}) \leq \frac{3!}{2}({\rm Vol}(P_{\alpha_5})  +{\rm Vol}(P_{\alpha_6})  + {\rm max}_{1\leq j \leq 4} {\rm Vol}(P_{\alpha_j} ) ) = 120  \,. $$
Thus there is no destabilizing subsheaf of rank $2$. 

There is a rank $3$  equivariant subsheaf $\mc{F}$ of $TX$ with associated matrix $\Lambda_{3\times 6 } = (\lambda_{ij})$ such 
that $\lambda_{1j} = -1$ for $1\leq j  \leq 4$, and all other $\lambda_{ij} = 0$. For every $4$-dimensional cone $\sigma_l$, $F^{\sigma_l}$ is generated over $k[S_{\sigma_l}]$ by $\partial_{\chi(m_i^{\sigma_l})}$, $1 \leq i \leq 3$.
The generators for different cones $\sigma_l$ and $\sigma_m$ are related by Jacobians of corresponding transition maps which are naturally
regular on $X_{\sigma_l\cap \sigma_m} $. Moreover, as $\chi(m_4^{\sigma_l})$ is either $t_4$ or $t_4^{-1}$, it follows easily that 
$\partial_{\chi(m_i^{\sigma_l})}$, $1 \leq i \leq 3$, is always a combination of $\partial_{t_1}$,  $\partial_{t_2}$ and $\partial_{t_3}$. 
Thus the subsheaf $\mc{F}$ indeed exists.
 It has the  maximum slope among rank 3 subsheaves, and the slope is
$$ \mu(\mc{F}) = \frac{3!}{3} ({\rm Vol}(P_{\alpha_1} )+ {\rm Vol}(P_{\alpha_2} )+{\rm Vol}(P_{\alpha_3} )+{\rm Vol}(P_{\alpha_4} )) = 128 =
 \mu(TX) \,.$$
This concludes the proof.
\end{proof}	

Here is the classification of all nonsingular Fano toric $4$-folds with Picard number $\leq 2$ according to stability with respect to the anticanonical polarization. 

\begin{theorem}\label{fano4}
	Suppose $X$ is one of the Fano toric $4$--folds with Picard number $\leq2$. Then
	\begin{enumerate}
		\item  $T\mathbb{P}^{4}$ is stable,
		\item  $TB_4$ and $TC_4$ are polystable,
		\item $TB_5$ is strictly semistable,
		\item   $TB_1$, $TB_2$ and $TB_3$ are unstable,
		\item $TC_1$, $TC_2$ and $TC_3$ are unstable.
	\end{enumerate}
\end{theorem}

\begin{proof}
	\begin{enumerate}
		\item This is well-known. See Theorem \ref{projective spaces} for an alternative proof. 
		\item $B_4$ and $C_4$ admit Einstein-K\"{a}hler metrics (see \cite[Theorem 3.4]{Nak2}).  Hence, these are  polystable.
		\item See Theorem \ref{detailed example of 4-fold}. 
		\item These are special cases of  Theorem \ref{family}.
		\item  $TC_1$ and $TC_2$ are destabilized by rank 2 subsheaves, and  $TC_3$ is destabilized by a rank 1 subsheaf. 
		These may be verified  by following a similar approach as in the proof of Theorem \ref{detailed example of 4-fold}.
	\end{enumerate}
\end{proof}

\section*{Acknowledgements}

We would like to thank the anonymous referee for many useful comments and suggestions that have helped to improve the exposition.  
The first-named author is supported by a J. C. Bose fellowship. The second-named author is supported by a SERB MATRICS research grant. The third-named author is supported by Narodowe Centrum Nauki 2018/30/E/ST1/00530. The last-named author has been supported in part by an SRP grant from METU NCC and a SERB MATRICS grant.


\begin{thebibliography}{BPV84}

\bibitem{At2} M. F. Atiyah: Complex analytic connections in fibre bundles,
{\it Trans. Amer. Math. Soc.} {\bf 85} (1957), 181--207.

\bibitem{Au} T. Aubin: \'Equations du type Monge-Amp\`ere sur les vari\'et\'es k\"ahl\'eriennes
compactes, {\it Bull. Sci. Math.} {\bf 102} (1978), 63--95.

\bibitem{BCCM} C. Birkar, P. Cascini, C. D. Hacon and J. McKernan: Existence of minimal models for varieties of log general type,
{\it J. Amer. Math. Soc.} {\bf 23} (2010), 405--468.

\bibitem{Bat} V. V. Batyrev: Dual polyhedra and mirror symmetry for Calabi-Yau hypersurfaces in toric varieties, 
{\it J. Algebraic Geom.} {\bf 3} (1994), 493--535. 

\bibitem{Bat2} V. V. Batyrev: On the classification of toric Fano 4-folds, {\it Jour. Math. Sci.} {\bf 94} (1999), 1021--1050.

\bibitem{BP} I. Biswas and A. J. Parameswaran: On the equivariant reduction of structure group of a principal bundle to a Levi subgroup,
{\it J. Math. Pures Appl.} {\bf 85} (2006), 54--70.

\bibitem{BDL} I. Biswas, S. Dumitrescu and M. Lehn, On the stability of flat complex vector bundles over
parallelizable manifolds, {\it Com. Ren. Math. Acad. Sci. Paris} {\bf 358} (2020), 151--158.


\bibitem{Fah} R. Fahlaoui: Stabilit\'e du fibr\'e tangent des surfaces de del Pezzo,
{\it Math. Ann.} {\bf 283} (1989), 171--176. 

\bibitem{Ful} W. Fulton: {\it Introduction to Toric Varieties}, Annals of Mathematics Studies,
131, Princeton University Press, Princeton, 1993.

\bibitem{HL} D. Huybrechts, and M. Lehn: {\it The geometry of moduli spaces of sheaves}, Aspects of Mathematics, E31. 
Friedr. Vieweg \& Sohn, Braunschweig, 1997. xiv+269 pp. ISBN: 3-528-06907-4 

\bibitem{Kan} T. Kaneyama: On equivariant vector bundles on an almost homogeneous
variety, {\it Nagoya Math. Jour.} {\bf 57} (1975), 65--86.

\bibitem{Klei} P. Kleinschmidt: A classification of toric varieties with few generators, {\it Aequationes Math.} {\bf 35} (1988),
254--266.

\bibitem{Kly} A. A. Klyachko: Equivariant bundles over toric varieties,
{\it Izv. Akad. Nauk SSSR Ser. Mat.} {\bf 53} (1989), 1001--1039,
1135; translation in Math. USSR-Izv. {\bf 35} (1990), no. 2, 337--375.

\bibitem{Kob} S. Kobayashi: Differential Geometry of Complex Vector Bundles, Publications of the Mathematical Society of Japan, {\bf 15}, Princeton University Press, Princeton, 1987. 

\bibitem{Kool1} M. Kool: {\it Moduli spaces of sheaves on toric varieties}, PhD thesis, University of Oxford, 2010.

\bibitem{Kool2} M. Kool: Fixed point loci of moduli spaces of sheaves on toric varieties, {\it Adv. Math.} {\bf 227} (2011), 1700--1755.

\bibitem{Lub} M. L\"ubke: Stability of Einstein-Hermitian vector bundles, {\it Manuscripta Math.}
{\bf 42} (1983), 245--257.

\bibitem{Nak1} Y. Nakagawa: Einstein-K\"ahler toric Fano fourfolds, {\it Tohoku Math. Jour.} {\bf 45} (1993), 297--310.

\bibitem{Nak2} Y. Nakagawa: Classification of Einstein-K\"{a}hler toric Fano fourfolds, {\it Tohoku Math. Jour.}
{\bf 46} (1994), 125--133. 

\bibitem{Oda} T. Oda: {\it Convex Bodies and Algebraic Geometry An introduction to the theory of 
toric varieties}, Translated from the Japanese, Ergebnisse der Mathematik und ihrer Grenzgebiete 
(3), 15. Springer-Verlag, Berlin, 1988.

\bibitem{Per} M. Perling: Graded rings and equivariant sheaves on toric varieties, {\it Math. Nachr.} {\bf 263/264} (2004), 181--197.

\bibitem{Stef} A. Steffens: On the stability of the tangent bundle of Fano manifolds, {\it Math. 
Ann.} {\bf 304} (1996), 635--643.

\bibitem{Ya} S.-T. Yau: On the Ricci curvature of a compact K\"ahler manifold and the complex 
Monge-Ampère equation. I, {\it Comm. Pure Appl. Math.} {\bf 31} (1978), 339--411.
 

\end{thebibliography}
\end{document}